\documentclass{amsart}

\usepackage{amsmath}
\usepackage{amssymb}
\usepackage{mathrsfs}
\usepackage{tikz}
\usepackage{amsthm}
\usepackage{hyperref}

\usepackage{tikz-qtree}

\newtheorem{thm}{Theorem}[section]

\newtheorem{prop}[thm]{Proposition}

\newtheorem{lem}[thm]{Lemma}

\newtheorem{cor}[thm]{Corollary}

\newtheorem*{claimstar}{Claim}

\theoremstyle{remark}
\newtheorem{rem}[thm]{Remark}

\theoremstyle{definition}
\newtheorem{defn}[thm]{Definition}

\newtheorem{exam}[thm]{Example}

\def\Ind#1#2{#1\setbox0=\hbox{$#1x$}\kern\wd0\hbox to 0pt{\hss$#1\mid$\hss}
	\lower.9\ht0\hbox to 0pt{\hss$#1\smile$\hss}\kern\wd0}

\def\notind#1#2{#1\setbox0=\hbox{$#1x$}\kern\wd0
	\hbox to 0pt{\mathchardef\nn=12854\hss$#1\nn$\kern1.4\wd0\hss}
	\hbox to 0pt{\hss$#1\mid$\hss}\lower.9\ht0 \hbox to 0pt{\hss$#1\smile$\hss}\kern\wd0}

\DeclareMathOperator{\prob}{Prob}
\DeclareMathOperator{\op}{op}
\DeclareMathOperator{\opR}{opR}
\DeclareMathOperator{\test}{test}

\begin{document}

\title{Model theory and combinatorics of banned sequences}
\author{Hunter Chase}
\address{Department of Mathematics, UIC, Chicago IL}
\email{hchase2@uic.edu}

\author{James Freitag}
\address{Department of Mathematics, UIC, Chicago IL}
\email{freitagj@gmail.com}

\thanks{James Freitag was supported by NSF grant no. 1700095}

\begin{abstract} 
We set up a general context in which one can prove Sauer-Shelah type lemmas. We apply our general results to answer a question of Bhaskar \cite{bhaskar2017thicket} and give a slight improvement to a result of Malliaris and Terry \cite{MT2017}. We also prove a new Sauer-Shelah type lemma in the context of $\op$-rank, a notion of Guingona and Hill \cite{guingonahill2015oprank}.
\end{abstract} 

\maketitle
	

\section{Introduction} 
A single combinatorial notion called \emph{VC-dimension} determines important dividing lines in both machine learning (PAC learnability of a class) and model theory (the independence/non-independence dichotomy, IP/NIP) \cite{laskowski1992vapnik}, and the finiteness of this quantity plays an essential role in the development of various structural results in theories without the independence property and in machine learning. Often at the root of these developments is the Sauer-Shalah Lemma, which for a formula $\phi(x;y)$ without the independence property, gives a polynomial bound on the shatter function associated with $\phi$---that is, the number of consistent $\phi$-types over finite sets. Without NIP, however, the number of $\phi$-types can grow exponentially in the size of the finite parameter set. In a recent paper, Bhaskar \cite{bhaskar2017thicket} noticed that when the formula $\phi$ is actually stable, that is, $\phi$ has finite Shelah $2$-rank (also called thicket dimension and Littlestone dimension in the context of set systems), one can relax the way in which the $\phi$-types are constructed, allowing for trees of parameters (explained below) while still proving polynomial bounds on the resulting collection of consistent $\phi$-types. Again, in the absence of stability the number of types formed in this manner can grow exponentially in the height of the tree. Following Bhaskar, we refer to this growth dichotomy theorem as the Thicket Sauer-Shelah Lemma. In \cite{MLMT}, we notice that stability also determines an important dividing line in machine learning; stability determines learnability in various settings of \emph{online learning}. In these settings of learning, various results at their core rely on the polynomial growth of the thicket shatter function.

In both settings described above, the growth of the number of types being polynomially bounded or exponential is completely determined by whether a simple combinatorial notion of dimension is finite, and the upper bound (which is tight in general) on the number of such types (in terms of the appropriate notion of dimension) is identical in both cases. In light of this, Bhaskar naturally asks if there is a single combinatorial principle which explains both the Sauer-Shalah Lemma and the (thicket) stable version. The main purpose of our paper is to set up a general context in which one can prove Sauer-Shelah type results into which both of the above contexts fit, answering Bhaskar's question as well as proving new results. Our solution to the problem is quite general and deals with what we call \emph{banned sequence problems}. 

Our general setup of banned sequence problems is an interesting combinatorial setting in its own right, and we will roughly describe the simplest context here. Suppose that you consider the collection of all binary sequences of length $n$, and for each subset of the indices of size $k$, there is at least one ``banned subsequence" of length $k$. How many binary sequences of length $n$ are there which avoid each of the banned sequences on all subsets of the indices of size $k$? Subject to some very mild assumptions on how the banned sequences are chosen, we show that there are at most 
\[
	\sum _{i=0}^{k-1} \binom{n}{i}
\] 
such sequences. This bound is the bound of the Sauer-Shelah Lemma. Without the mild assumptions, we show that this bound can be violated. The generality of our setup covers both the settings mentioned above as well as yielding some new results.
 
We give a slight improvement of a result of Malliaris and Terry \cite{MT2017} regarding sizes of cliques and independent sets in stable graphs. Essentially, their result uses the finiteness of a certain combinatorial dimension, \emph{tree rank}, in order to establish polynomial bounds strong enough to get a version of the Erdos-Hajnal conjecture, among other results (Malliaris and Terry also develop further structural properties of graphs which we will not touch on in this paper). We examine tree rank in the general context of banned sequence problems, and as a result, give a slight improvement to their bounds. Following this, we give an adaptation of the VC-theorem to the setting of finite thicket dimension set systems. The VC-theorem roughly says that for a set system of finite VC dimension on a probability space, a random sampling by a large enough tuple will, with high probability, give a good estimate of the measures of all of the sets in the set system (i.e. the proportion of elements in the tuple belonging to a set is close to the measure of that set). We show that using the stronger hypothesis of finite thicket dimension allows for a loosening of the sampling assumptions, allowing for sampled elements to depend on the outcomes of the previous samples.

In the last part of the paper, we turn to the setting of $\op$-ranks. For each $s \in \mathbb{N}$, Guingona and Hill \cite{guingonahill2015oprank} define a rank of partial types, $\op_s$-rank. For instance, when $s=1$, $\op_1$-rank is equal to the Shelah 2-rank. Working with set systems of finite $\op_s$-rank, we establish a new variant of the Sauer-Shelah Lemma using our banned sequence setup. 

We note that not every known variant of the Sauer-Shelah Lemma seems to fit into the context of banned sequence problems; the main results of \cite{chernikov2014n} establish a variant of Sauer-Shelah for $n$-dependent theories which does not seem to easily fit into our context of banned sequence problems. Is there a general setup which also covers the known Sauer-Shelah style results for $n$-dependent theories? This seems reasonable to ask because $n$-dependent theories generalize NIP theories in a way similar to how theories with finite $\op_s$-rank generalize stable theories. 

\subsection{Organization} 
In section \ref{prelim}, we give the necessary preliminary notation for our results. In section \ref{banned}, we lay out the basic theory of banned sequence problems along with some applications. In section \ref{VCVCVC}, we adapt the VC-theorem. In section \ref{higherdim}, we generalize our banned sequence problems. In section \ref{opsection}, we apply generalized banned problems to the $\op$-rank setting. 

\subsection{Acknowledgements} 
The authors would like to thank Dave Marker, Vincent Guingona, Caroline Terry, Maryanthe Malliaris, Gabriel Conant, and Cameron Hill for useful conversations and suggestions during the preparation of this article. The authors would also like to especially thank Alex Kruckman and Siddharth Bhaskar for a number of useful comments as well as pointing out some mistakes in early versions of several proofs.

\section{Preliminaries} \label{prelim} 

Our primary combinatorial tool applies to theorems surrounding VC-dimension and thicket dimension (also known as Shelah's 2-rank in model theory or Littlestone dimension in machine learning), and we recall those definitions and relevant theorems. The next several definitions can be found in various sources, e.g. \cite{simon2015guide}. 

Throughout, any indexing starts at 0, and $[n] := \{0, 1, \ldots, n-1 \}$. By $\binom{[n]}{k}$ we mean the collection of all subsets of $[n]$ of size $k$.

Recall that a set system $(X, \mathcal{F})$ (often referred to as $\mathcal{F}$ when $X$ is understood) consists of a set $X$ and a collection $\mathcal{F} \subseteq \mathcal{P}(X)$ of subsets of $X$. For $Y \subseteq X$, the projection of $\mathcal{F}$ onto $Y$ is the set system with base set $Y$ and collection of subsets
\[
	\mathcal{F}_Y := \{ F \cap Y \, | \, F \in \mathcal{F} \}.
\]

VC-dimension measures the ability of a set system to pick out subsets of a set of a given size.

\begin{defn}
	A set system $(X, \mathcal{F})$ \emph{shatters} a set $Y$ if $\mathcal{F}_Y = \mathcal{P}(Y)$. The VC-dimension of $\mathcal{F}$ is the largest $k < \omega$ such that $\mathcal{F}$ shatters some set of size $k$, or is infinite if $\mathcal{F}$ shatters arbitrarily large sets. The \emph{shatter function}
	\[
		\pi_\mathcal{F}(n) := \sup_{Y \subseteq X, |Y| = n} |\mathcal{F}_Y|
	\]
	is given by the supremum of the size of the projection onto subsets of a given size.
\end{defn}

If a set system has finite VC-dimension, then we obtain a polynomial bound on the shatter function.

\begin{thm}[Sauer-Shelah Lemma] \label{VC-SS}
	Let $\mathcal{F}$ be a set system of VC-dimension $k$. Then the maximum size of a projection from $\mathcal{F}$ onto a set $A = \{a_0, \ldots, a_{n-1}\}$ of size $n$ is $\sum_{i=0}^k \binom{n}{i}$. In particular,
	\[
		\pi_\mathcal{F}(n) \leq \sum_{i=0}^k \binom{n}{i}.
	\]
\end{thm}

Several proofs of the Sauer-Shelah Lemma can be found in various sources, e.g. \cite{simon2015guide, ngo2010SSproofs}. 

Thicket dimension is a variant of VC-dimension; our development follows \cite{bhaskar2017thicket}. Given a set from the set system, elements are presented sequentially, with the element presented depending on membership of previous elements.

\begin{defn}
	A \emph{binary element tree} of height $n$ with labels from $X$ is a function $T: 2^{<n} \rightarrow X$. A \emph{node} is a binary sequence $\sigma \in 2^{<n}$ along with its label, $a_\sigma := T(\sigma)$. A \emph{leaf} is a binary sequence of length $n$, $\tau: [n] \rightarrow \{0,1\}$. A leaf $\tau$ is properly labeled by a set $A$ if for all $m < n$, 
	\[
		a_{\tau|_{[m]}} \in A \quad \text{iff} \quad \tau(m) = 1.
	\]
\end{defn}

\begin{defn}
	The \emph{thicket dimension} of a set system $(X, \mathcal{F})$ is the largest $k < \omega$ such that there is a binary element tree of height $k$ with labels from $X$ such that every leaf can be properly labeled by elements of $\mathcal{F}$, or is infinite if there are such trees of arbitrary height. The \emph{thicket shatter function} $\rho_\mathcal{F}(n)$ is the maximum number of leaves properly labeled by elements of $\mathcal{F}$ in a binary element tree of height $n$.
\end{defn}

\begin{thm}[Thicket Sauer-Shelah Lemma \cite{bhaskar2017thicket}] \label{thicketSS} 
	Let $\mathcal{F}$ be a set system of thicket dimension $k$. Then the maximum number of properly labeled leaves in a binary element tree of height $n$ is $\sum_{i=0}^k \binom{n}{i}$. In particular,
	\[
		\rho_\mathcal{F}(n) \leq \sum_{i=0}^k \binom{n}{i}.
	\]
\end{thm}

VC-dimension and the VC shatter function can be viewed in the context of binary element trees where every node of the same height is labeled with the same element, i.e. $a_\sigma = a_{\sigma'}$ whenever $|\sigma| = |\sigma'|$.

There are dual notions of both VC-dimension and thicket dimension, and their corresponding shatter functions, where the roles of elements and sets are reversed. 

\begin{defn}
	Given a set system $(X, \mathcal{F})$, the dual set system $(X,\mathcal{F})^*$, or just $\mathcal{F}^*$, is the set system with base set $\mathcal{F}$ where the subsets are given by 
	\[
		\{ F \, | \, F \in \mathcal{F}, x \in F \}
	\]
	for each $x \in X$. The dual VC (resp., thicket) dimension of $\mathcal{F}$ is the VC (resp., thicket) dimension of $\mathcal{F}^*$.
\end{defn}

Dual thicket dimension can be calculated by examining \emph{binary decision trees}, where nodes are labeled by sets in the set system, and leaves are labeled by elements. Dual VC-dimension can be calculated similarly.

In model theory, given a model $\mathcal{M}$, the VC (resp., thicket) dimension of a partitioned formula $\phi(x;y)$ is the VC (resp., thicket) dimension of the set system
\[
	(M^{|x|}, \{\phi(M^{|x|}, b) \, | \, b \in M^{|y|} \}).
\]
These combinatorial notions encode model-theoretic dividing lines. A formula is NIP iff it has finite VC-dimension, and is stable iff it has finite thicket dimension. 

\section{The combinatorics of banned sequences} \label{banned} 

The binary element tree structure used to define thicket dimension allows us to identify a leaf of the tree with the binary sequence corresponding to the path through the tree to that leaf. Then counting properly-labeled leaves amounts to counting the corresponding binary sequences. We establish a framework for counting binary sequences under certain conditions reflecting the tree structure, from which we will obtain a unified proof of the Sauer-Shelah lemmas. 

\subsection{Banned binary sequences and Sauer-Shelah lemmas}

Our framework for counting binary sequences will reflect the height of the tree as well as the dimension (either thicket or VC) of the set system. We find it easier to count banned sequences. Having thicket dimension $k-1$ says that in a tree of height $k$, there are some leaves which cannot be properly labeled, and those leaves correspond to sequences that we ban.

\begin{defn}
	A \emph{$k$-fold banned binary sequence problem (BBSP) of length $n$}, $g$, is a function 
	\[
		g: \binom{[n]}{k} \times 2^{n-k} \rightarrow \mathcal{P} (2 ^k) \setminus \{\emptyset \}.
	\]
\end{defn}

Informally, for each $k$-subset of $[n]$ and each binary sequence of length $n-k$, we pick at least one binary sequence of length $k$ to ban. Sometimes we will refer to the sequences selected by the function $f$ as \emph{banned subsequences.} 

\begin{rem} \label{noteymark} 
	It will be convenient to view binary sequences as functions, where the domain is the appropriate set of indices. Given $S \in \binom{[n]}{k}$, when we consider $f(S, X)$, we view $X \in 2^{n-k}$ as a function $X: [n] \setminus S \rightarrow \{0,1\}$, and elements of $f(S, X)$ as functions $Z: S \rightarrow \{0,1\}$.
	
	We shall denote the union of two binary sequences $X$ and $Y$ with disjoint domains as $X \wedge Y$. For example, if $X$ has domain $\{0,2\}$, with $X(0) = X(2) = 0$, and $Y$ has domain $\{1\}$ with $Y(1) = 1$, then $X \wedge Y$ is the binary sequence 010. When we wish to extend a sequence by appending some $j \in \{0,1\}$, we will merely write $X \wedge j$, with the index of $j$ usually understood from the context.
	
	For a fixed $S \in \binom{[n]}{k}$, we denote the elements of $S$ by $\{s_0, \ldots , s_{k-1} \}$, where $s_0 <s_1 < \ldots < s_{k-1} $. 	
\end{rem} 

\begin{defn} 
	A \emph{solution} to a $k$-fold banned binary sequence problem of length $n$, $f$ is a binary sequence, $X \in 2^n$ such that for any $S \subseteq [n]$, 
	\[
		X|_S \notin f(S, X|_{[n] \setminus S} ).
	\]
	A sequence which is not a solution is \emph{banned}.
\end{defn} 

Intuitively, a solution to a banned binary sequence problem is a sequence which avoids every banned subsequence. In applications to binary element trees, properly labeled leaves will correspond to solutions of a certain banned binary sequence problem.

Without further assumptions, the number of solutions of a BBSP can grow exponentially in $n$ for a fixed $k$.

\begin{prop}
	A $k$-fold BBSP of length $n$, $f$, has at most $(2^k - 1)2^{n-k}$ solutions.
\end{prop}

\begin{proof}
	Fix $S \in \binom{n}{k}$. For $X: [n] \setminus S \rightarrow \{0,1\}$ and $Z: S \rightarrow \{0,1\}$, $X \wedge Z$ can only be a solution if $Z \notin f(S, X)$, and for each of $2^{n-k}$ many such $X$'s, there are at most $2^k - 1$ many $Z$'s.
\end{proof}

We observe that to obtain this bound, and so have only $2^{n-k}$ banned sequences, we must be able to find a collection $\mathcal{B}$ of $2^{n-k}$ sequences $Y: [n] \rightarrow \{0,1\}$ such that for all $T \in \binom{[n]}{n-k}$ and all $X: T \rightarrow \{0,1\}$, there is some $Y \in \mathcal{B}$ such that $X \subseteq Y$. Then we set $f([n] \setminus T, X) := \{Y|_{[n] \setminus T}\}$. In general this is not possible. It is possible for $k = n$, where we simply pick a sequence of length $n$ to ban, $k = n-1$, where $\mathcal{B}$ can consist of, say, the two constant sequences, and $k = 1$, given below. But this condition already cannot be met for $k = 2$ and $n = 4$. In this case, one can verify that the minimum size of $\mathcal{B}$ to satisfy the above condition is 5, and so a 2-fold BBSP of length 4 can have at most 11 solutions.

\begin{exam}
	Let $f$ be the 1-fold BBSP of length $n$ given by
	\[
		f(\{s\}, X) = \begin{cases}
			1 & X \text{ has an even number of 1s} \\
			0 & X \text{ otherwise.}
		\end{cases}
	\]
	Then $f$ has $2^{n-1}$ solutions, given by those binary sequences which have an even number of 1s.
\end{exam}

We therefore need stronger hypotheses in order to bound the number of solutions by the Sauer-Shelah bound.

\begin{defn} 
	A $k$-fold banned binary sequence problem of length $n$, $f$, is \emph{not hereditary} if there is $S \in \binom{[n]}{k}$ such that for all $Z_\alpha: S \rightarrow \{0,1\}$, there is $X_\alpha: [n] \setminus S \rightarrow \{0,1\}$ such that $Z_\alpha \notin f(S, X_\alpha)$, and additionally, for $Z_\alpha \neq Z_\beta$, the first index at which $X_\alpha \wedge Z_\alpha$ and $X_\beta \wedge Z_\beta$ differ is in $S$.
	
	Otherwise, say $f$ is \emph{hereditary}.
\end{defn}

Hereditary BBSPs reflect enough of the tree structure so as to obtain the Sauer-Shelah bound on the number of sequences as well as frame thicket and VC trees as hereditary BBSPs, and thus derive the corresponding Sauer-Shelah lemmas.

\begin{thm} \label{maincombin} Any hereditary $k$-fold banned binary sequence problem of length $n$ has at most $\sum _{i=0}^{k-1} \binom{n}{i}$ solutions.
\end{thm} 

\begin{proof} 
	We prove the result by induction on $n$ and $k$. Throughout, we fix a hereditary $k$-fold banned binary sequence problem of length $n$, $f$. Let $B(f)$ denote the number of sequences banned by $f$. It suffices to prove that
	\[
		B(f) \geq 2^n - \sum_{i=0}^{k-1} \binom{n}{i}.
	\]
	
	The base cases are $k = n$ and $k = 1$. When $k = n$, the result is immediate, since then $2^n - \sum _{i=0}^{k-1} \binom{n}{i} = 1$ and any BBSP has at least one banned sequence.  
	
	When $k=1$, we must show that $B(f) \geq 2^n - 1$, i.e. there is at most one solution. Assume for contradiction that $X \neq Y$ are two distinct solutions to $f$, with $s$ the first index at which $X$ and $Y$ differ, say with $X(s) = 0$ and $Y(s) = 1$. If $s = n-1$, then we have that $f(\{n-1\}, X|_{[n-1]}) = \emptyset$, a contradiction. If $s < n-1$, we observe that $0 \notin f(\{s\}, X_{[n] \setminus \{s\}})$ and $1 \notin f(\{s\}, Y_{[n] \setminus \{s\}})$. Since the first index at which $X$ and $Y$ differ is $s$, this witnesses that $f$ is not hereditary, a contradiction. So $B(f) \geq 2^n - 1$.
	
	Otherwise, we proceed by induction. Given a binary sequence $X$ of length $n$, we call an entry of $X$ \emph{bad} if it is part of a $k$-subset $S$ such that $X|_S \in f(S, X_{[n] \setminus S})$. 
	
	We call a binary sequence $X'$ of length $n-1$ \emph{potentially bad for $T \in \binom{[n-1]}{k-1}$} if there is $j \in \{0,1\}$ such that the $(n-1)$th entry\footnote{We emphasize that we index from 0, so that the $(n-1)$th entry will be the last entry of a sequence of length $n$.} of $X' \wedge j$ is bad for $S = T \cup \{n-1\}$. That is, setting $S = T \cup \{n-1\}$, we have $X'|_T \wedge j \in f(S, X'|_{[n-1] \setminus T})$. In particular, sequences which are potentially bad for $T$ can be extended to sequences which are banned by $f$, witnessed by $S = T \cup \{n-1\}$.
	
	Let $\hat{f}$ be the $(k-1)$-fold banned binary sequence problem of length $n-1$ given by those sequences which are potentially bad for some $T \in \binom{[n-1]}{k-1}$. That is, let $Z \in \hat{f} (T, X)$ if $Z \wedge X$ is potentially bad for $T$. We claim that $\hat{f}$ is hereditary. If $\hat{f}$ were not hereditary, then there would exist $T \in \binom{[n-1]}{k-1}$ such that for each $Z_\alpha: T \rightarrow \{0, 1\}$, there is $X_\alpha: [n-1] \setminus T \rightarrow \{0, 1\}$ such that $Z_\alpha \notin \hat{f}(T, X_\alpha)$, with the first index at which any two $X_\alpha \wedge Z_\alpha$ and $X_\beta \wedge Z_\beta$ differ belonging to $T$. Note that for some $\alpha$ and some $j \in \{0, 1\}$, we have that 
	\[
		Z_\alpha \wedge j \in f(T \cup \{n-1\}, X_\alpha),
	\]
	or else associating $Z_\alpha \wedge j$ with $X_\alpha$ would witness that $f$ itself is not hereditary. In particular, $Z_\alpha \wedge X_\alpha$ is potentially bad for $T$. Then, by definition of $\hat{f}$, $Z_\alpha \in \hat{f}(T, X_\alpha)$, a contradiction. So $\hat{f}$ is also hereditary.
	
	Let $f'$ be the $k$-fold banned binary sequence problem of length $n-1$ given by those sequences which contain a banned subsequence for $f$ in the first $n-1$ entries, for any choice of the last entry. That is, given $S \in \binom{[n-1]}{k}$, let $Z \in f'(S, X)$ iff $Z \in f(S, X \wedge j)$ for all $j \in \{0, 1\}$. We claim that $f'$ is hereditary. Suppose for contradiction that $f'$ is not hereditary. Then there is $S \in \binom{[n-1]}{k}$ such that for all $Z_\alpha: S \rightarrow \{0, 1\}$, there is $X_\alpha: [n-1] \setminus S \rightarrow \{0,1\}$ such that $Z_\alpha \notin f'(S, X_\alpha)$, with the first index at which any two $X_\alpha \wedge Z_\alpha$ and $X_\beta \wedge Z_\beta$ differ belonging to $S$. That is, by definition of $f'$, for each $\alpha$, there is $j_\alpha \in \{0,1\}$ such that $Z_\alpha \notin f(S, X_\alpha \wedge j_\alpha)$. Let $X'_\alpha$ be $X_\alpha \wedge j_\alpha$. Then associating $Z_\alpha$ with $X'_\alpha$ witnesses that $f$ is not hereditary, a contradiction.
	
	Now, we aim to prove
	\[
		B(f) \geq B(\hat{f}) + B(f').
	\]
	For a given sequence $X'$ which is banned by $\hat{f}$, there is at least one extension $X$ which is banned by $f$, and we pick one such an extension. For a given sequence $Y'$ banned by $f'$, at most one extension $Y$ was already obtained by extending a banned sequence $X'$ for $\hat{f}$. So there is at least one extension $Y$ which is banned by $f$ (by definition of $f'$) but was not obtained by extending banned sequences for $\hat{f}$. Therefore these banned sequences constructed from $f'$ and $\hat{f}$ have no common members, and so we have
	\[
		B(f) \geq  B(\hat{f}) + B(f'),
	\]
	as desired. By induction, we have that
	\[
		B(f) \geq \left( 2^{n-1} - \sum _{i=0}^{k-2} \binom{n-1}{i} \right) +  \left( 2^{n-1} - \sum _{i=0}^{k-1} \binom{n-1}{i} \right).
	\]
	Noting that $\binom{n-1}{i} + \binom{n-1}{i-1} = \binom{n}{i}$ whenever $i>0$, and that $\binom{n-1}{0} = \binom{n}{0}$, we see that 
	\[
		B(f) \geq 2^n - \sum _{i=0}^{k-1} \binom{n}{i}.
	\]
	Thus $f$ has at most $\sum_{i=0}^{k-1} \binom{n}{i}$ solutions.
\end{proof}

It shall be useful to identify a stronger banned binary sequence problem, namely those in which $f(S, X)$ depends only on $S$.

\begin{defn}
	A banned binary sequence problem $f$ is $\emph{independent}$ if $f(S, X) = f(S, Y)$ for any $X, Y: [n] \setminus S \rightarrow {0,1}$. When $f$ is independent, we write $f(S)$.
\end{defn}

\begin{cor}\label{Ind-SS}
	Any independent $k$-fold banned binary sequence problem of length $n$ has at most $\sum_{i=0}^{k-1} \binom{n}{i}$ solutions.
\end{cor}

\begin{proof}
	We check that $f$ is hereditary. If not, then there is $S \in \binom{[n]}{k}$ such that for all $Z_\alpha : S \rightarrow \{0,1\}$, there is $X_\alpha: [n] \setminus S \rightarrow \{0,1\}$ with $Z_\alpha \notin f(S, X_\alpha) = f(S)$. But then $f(S)$ is empty, a contradiction. The result follows from Theorem \ref{maincombin}.
\end{proof}

Banned binary sequence problems provide a common framework to prove Sauer-Shelah type bounds.

\begin{proof}[Proof of Theorem \ref{VC-SS}]
	We obtain a $k+1$-fold independent BBSP $f$ of length $n$ as follows. Given $S = \{a_{s_0}, \ldots a_{s_{k}}\} \in \binom{A}{k+1}$, let $f(S)$ be the set of binary sequences $Z$ of length $k+1$ such that there is no $F \in \mathcal{F}$ such that $a_{s_i} \in F$ iff $Z(i) = 1$. We have that $f(S)$ is nonempty since the VC-dimension of $\mathcal{F}$ is $k$, and $f$ is clearly independent. Then a subset $B$ of $A$ is in the projection from $\mathcal{F}$ onto $A$ iff the characteristic sequence of $B$ (i.e. the sequence where the $j$th entry is 1 iff $a_j \in B$) is a solution to $f$. The result follows from Corollary \ref{Ind-SS}.
\end{proof}

\begin{proof}[Proof of Theorem \ref{thicketSS}]
	Let $T$ be a binary element tree of height $n$, with nodes $a_\sigma$ for $\sigma \in 2^{<n}$. We obtain a $k+1$-fold hereditary BBSP of length $n$, $f$, as follows. Given $S = \{s_0, \ldots, s_k\} \in \binom{[n]}{k+1}$ where $s_0 < s_1 < \cdots < s_k$ and $X: [n] \setminus S \rightarrow \{0,1\}$, we obtain a binary element tree of height $k+1$ by taking all paths $\tau \in 2^n$ through $T$ such that $X \subseteq \tau$. Any two such paths first differ at some node $a_\sigma$ where $|\sigma| \in S$, so removing all other nodes gives us the binary element tree $T_{S,X}$ of height $k+1$. Since $\mathcal{F}$ has thicket dimension $k$, not all leaves of $T_{S,X}$ can be properly labeled, so let $f(S, X)$ be the set of all sequences whose corresponding leaves in $T_{S,X}$ cannot be properly labeled. Then a leaf in $T$ can only be properly labeled if the corresponding sequence is a solution to $f$.
	
	We now show that $f$ as constructed above is hereditary. Fix $S = \{s_0, \ldots, s_k\}$, and suppose for contradiction that this choice of $S$ witnesses that $f$ is \emph{not} hereditary. Then, for each $Z_\alpha: S \rightarrow \{0,1\}$, there is $X_\alpha: [n] \setminus S \rightarrow \{0,1\}$ such that $Z_\alpha \notin f(S, X_\alpha)$. We obtain a complete binary tree of height $k+1$ specified by each path $X_\alpha \wedge Z_\alpha$ constructed in this manner, restricted to $S$. In particular, any two paths constructed in this manner first differ at some index in $S$, as the first index at which $X_\alpha \wedge Z_\alpha$ and $X_\beta \wedge Z_\beta$ differ is in $S$. Since each $Z_\alpha$ is not banned, we have a complete binary tree of height $k+1$ in which every leaf can be properly labeled, a contradiction.
	
	The result then follows from Theorem \ref{maincombin}.
\end{proof}

\subsection{An application to type trees}

Banned binary sequence problems can be applied to other problems with a tree structure. We use this to improve a result of Malliaris and Terry \cite{MT2017}. 

\begin{defn} 
	Given a graph $G = (V,E)$ on $n$ vertices and $A \subseteq 2^{<n}$, closed under initial segments, we say that a labeling $V= \{ a_\eta \, | \, \eta \in A \}$ is a \emph{type tree} if for each $ \eta \in A:$
	\begin{enumerate} 
		\item If $\eta \wedge 0 \in A$, then $ a_{\eta \wedge 0}$ is nonadjacent to $a_ \eta$. If $\eta \wedge 1 \in A$, then $ a_{\eta \wedge 1}$ is adjacent to $a_ \eta$. 
		\item If $\eta \subsetneq \eta' \subsetneq \eta''$, then $a_\eta $ is adjacent to $a_{\eta'} $ if and only if $a_\eta$ is adjacent to $a_{\eta''}$. 
	\end{enumerate}
	A type tree has height $h$ if $A \subseteq 2^{<h}$ but $A \nsubseteq 2^{<h-1}$.
\end{defn}

This is a specific instance of a type tree. More generally, given a model $\mathcal{M}$, a finite set $B \subseteq M$, a finite collection $\Delta$ of partitioned formulas closed under cycling of the variables, and $A\subseteq \omega^{<\omega}$ closed under initial segments, a type tree is a labeling $B = \{b_\eta \, | \, \eta \in A\}$ such that, for any $\eta, \eta' \in A$, $b_\eta$ and $b_{\eta'}$ have the same $\Delta$-type over their common predecessors $\{b_{\zeta} \, | \, \zeta \subsetneq \eta, \beta \subsetneq \eta' \}$ iff $\eta \subseteq \eta'$ or $\eta' \subseteq \eta$. Type trees are used in more generality in \cite{MS2014stableregularity}.

\begin{defn}
	The \emph{tree rank} of a graph $G = (V,E)$ is the largest integer $t$ such that there is a subset $V' \subset V$ and some indexing $V' =  \{ a_ \eta \, | \, \eta \in 2^{<t} \}$ which is a type tree for the induced graph on $V'$, i.e. the type tree of $V'$ is a full binary tree of height $t$.
\end{defn} 

The main interest in type trees for graphs lies in the fact that if we have a branch of length $h$ for a graph $(V,E)$ with tree rank $t$, there is a clique or independent set of size at least $\max \{\frac{h}{2}, t\}$ \cite[Lemma 4.4]{MT2017}. More generally, branches through a type tree can be used to extract indiscernible sequences \cite[Theorem 3.5]{MS2014stableregularity}. In both cases, stability establishes the length of long branches through the type tree. For graphs, this is by way of tree rank---observe that the edge relation having thicket dimension $k$ implies that the tree rank is at most $k+1$. We use banned binary sequence problems to improve the bounds from \cite[Theorem 4.6]{MT2017}. The improvement is modest, but it demonstrates how banned binary sequence problems accommodate the combinatorics of type trees, at least in the case of the graph edge relation.

\begin{thm}\label{treeheight}
	Let $G = (V,E)$ be a graph with $|V| = n$ and tree rank $t \geq 2$. Suppose $A \subseteq 2^{< n}$ and $V = \{a_\eta \, | \, \eta \in A \}$ is a type tree with height $h$, where $h \geq 2t$. Then
	\[
		h \geq \left( n \cdot (t-2)! \right)^{\frac{1}{t}} + 1.
	\]
\end{thm}

The assumptions on $t$ and $h$ are not restrictive if our aim is to obtain cliques or independent sets. If $t = 1$, then there is no branching, and we obtain a clique or independent set of size $\frac{n}{2}$. If $h < 2t$, then the largest clique or independent set guaranteed by \cite[Lemma 4.4]{MT2017} is just the tree rank $t$.

\begin{proof} 
	We will associate a hereditary $t$-fold banned binary sequence problem of length $h-1$ with the type tree. Fix any subset $S = \{s_0, \ldots, s_{t-1}\}$ in $\binom{[h-1]}{t}$ and any $X \in 2^{[h-1] \setminus S}$. Let $f(S, X)$ consist of all $Z \in 2^S$ such that $(X \wedge Z) |_{[s_{t-1} + 1]}$ is an element of $2^{<h}$ which is not in the index set $A$ of the type tree. 
	
	
	Suppose for contradiction that $f(S,X) = \emptyset$. For each $\eta \in 2^{<t+1}$, we identify $\eta$ with a partial function $Z_\eta : S \rightharpoonup \{0,1\}$, where $\eta(i) = Z_\eta(s_i)$. For each $i < t$ and each $\eta: [i] \rightarrow \{0,1\}$ in $2^{<t+1} \setminus 2^t$, let $b_\eta = a_{(X \wedge Z_\eta) |_{[s_i]}}$. For each $\eta: [t] \rightarrow \{0,1\}$ in $2^t$, let $b_\eta = a_{(X \wedge Z_\eta) |_{[s_{t-1} + 1]}}$. Note that each $b_\eta$ is well-defined---in particular, for $\eta \in 2^t$, if $b_\eta = a_{(X \wedge Z_\eta) |_{[s_{t-1} + 1]}}$ was not an element of the type tree, then we would have $Z_\eta \in f(S, X)$. The rest of the elements are well-defined since the index set of a type tree is closed under initial segments. Then the $b_\eta$ define a full binary type tree of height $t+1$, contradicting our assumption that the tree rank of $G$ is $t$. So $f$ is a $t$-fold BBSP of length $h-1$.
	
	We check that $f$ is hereditary. Suppose for contradiction that $f$ is not hereditary, witnessed by some $S \in \binom{[h-1]}{t}$. So for each $Z_\alpha: S \rightarrow \{0,1\}$, there is $X_\alpha: [h-1] \setminus S \rightarrow \{0,1\}$ such that $Z_\alpha \notin f(S,X_\alpha)$, and for $\alpha \neq \beta$, the first index at which $X_\alpha \wedge Z_\alpha$ and $X_\beta \wedge Z_\beta$ differ is in $S$. Identify each $\eta \in 2^{<t+1}$ with $Z_\eta$ as above. For each $i < t$  and each $\eta: [i] \rightarrow \{0,1\}$, let $b_\eta = a_{(X_\alpha \wedge Z_\alpha)|_{[s_i]}}$ for any $\alpha$ such that $Z_\eta \subseteq Z_\alpha$.   For each $\eta: [t] \rightarrow \{0,1\}$, let $b_\eta = a_{(X_\eta \wedge Z_\eta) |_{[s_{t-1} + 1]}}$. These $b_\eta$ are defined since $Z_\eta \in f(S, X_\eta)$ by hypothesis. All other $b_\eta$, for $\eta: [i] \rightarrow \{0,1\}$, $i < t$, are defined since type trees are closed under initial segments, and well-defined since if $Z_\eta \subseteq Z_\alpha, Z_\beta$, then the first index at which $X_\alpha \wedge Z_\alpha$ and $X_\beta \wedge Z_\beta$ differ is in $S$ and is at least $s_i$. Then the $b_\eta$ form a type tree of height $t+1$, a contradiction.
	
	Thus a type tree of height $h$ gives a hereditary $t$-fold banned binary sequence problem of length $h-1$. Now, by Theorem \ref{maincombin}, the number of nodes at level $h_0$, $h_0 = 0, \ldots, h-1$, is at most 
	\[
		\sum _{i=0}^{t-1} \binom{h_0}{i}.
	\]
	Thus, the total number of nodes of a type tree of height $h$ and tree rank $t$ is at most 
	\begin{align}
		\sum_{h_0 = 0}^{h-1} \sum_{i=0}^{t-1} \binom{h_0}{i} & = 1 + \sum_{h_0 = 1}^{h-1} \sum_{i=0}^{t-1} \binom{h_0}{i} \nonumber \\
		& = 1 + \sum_{h_0 = 1}^{h-1} \left( 1 + \sum_{i=1}^{t-1} \binom{h_0}{i} \right ) \nonumber \\
		& \leq \sum_{h_0 = 1}^{h-1} \sum_{i=1}^{t-1} \binom{h-1}{i} \label{treeineq1} \\
		& \leq \sum_{h_0 = 1}^{h-1} \sum_{i=1}^{t-1} \frac{(h-1)^{t-1}}{(t-1)!} \label{treeineq2} \\
		& \leq \sum_{h_0 = 1}^{h-1} \frac{(h-1)^{t-1}}{(t-2)!} \nonumber \\
		& \leq \frac{(h-1)^t}{(t-2)!}, \nonumber
	\end{align}
	where estimates in (\ref{treeineq1}) and (\ref{treeineq2}) follow from hypotheses on $t$ and $h$. Then
	\[
		\frac{(h-1)^t}{(t-2)!} \geq n,
	\]
	so
	\[
		h \geq \left( n \cdot (t-2)! \right)^{\frac{1}{t}} + 1.
	\]
\end{proof} 

Under the hypotheses of Theorem \ref{treeheight}, applying \cite[Lemma 4.4]{MT2017} gives us a clique or independent set of size at least 
\[
	\frac{\left( n \cdot (t-2)! \right)^{\frac{1}{t}} +1}{2}.
\]
This is an improvement of the lower bound given by Malliaris and Terry \cite[Corollary 4.7]{MT2017}. 

\section{Adapting the Fundamental Theorem of Vapnik and Chervonenkis} \label{VCVCVC} 

The fundamental theorem of Vapnik and Chervonenkis, also known as the VC-theorem, states that for a set system of finite VC dimension, a random sampling by a single tuple will, with high probability, uniformly give a good estimate of the measures of all of the sets in the set system. We adapt this to the finite thicket dimension case. The stronger hypothesis of finite thicket dimension allows for more complicated sampling, by allowing for sampled elements to depend on the outcomes of the previous samples.

Let $\mathcal{F}$ be a set system on $X$, with $S \in \mathcal{F}$. Fix $\bar{x} = (x_\sigma)_{\sigma \in 2^{<n}}$, which we call a \emph{test tree}. We define 
\[
	\test(\bar{x}, S) := \frac{1}{n} \sum_{i=0}^{n-1} \sum_{\tau \in 2^i} \left( 1_S( x_\tau ) \cdot \prod_{\sigma < \tau} \chi_S^\tau( x_\sigma) \right)
\]
where $1_S$ is the indicator function for $S$, $\sigma < \tau$ means $\sigma$ is an initial segment of $\tau$, and $\chi_S^\tau$ is the $\tau$-indicator function for $S$, i.e. for $\sigma < \tau$,
\[
	\chi_S^\tau(x_\sigma) = \begin{cases}
		1 & x_\sigma \in S \leftrightarrow \tau(|\sigma|) = 1 \\
		0 & \text{otherwise.}
	\end{cases}
\]
This is calculated by first checking if $x_{S, 0} := x_\emptyset \in S$, then checking if $x_{S, 1} := x_{(1_S(x_{S, 0}))} \in S$, then checking if $x_{S, 2} := x_{(1_S(x_{S_0}), 1_S(x_{S, 1}))} \in S$, and so on. That is, we follow the path determined by $S$ through the test tree and check each element on the path. Indeed, we have that $\prod_{\sigma < \tau} \chi_S^\tau( x_\sigma) = 1$ iff $\tau$ follows the characteristic path of $S$ on $(x_\sigma)$ (i.e. $\tau$ is an initial segment of the unique binary sequence $\eta \in 2^n$ such that, for each $\sigma < \eta$, $x_\sigma \in S$ iff $\eta(|\sigma|) = 1$), and checking $x_\tau$ is only relevant if we have done so. Following this path, we count how many positive answers we get and divide by the number of queries $n$ to estimate the measure of $S$. (We think of $\test(\bar{x}, S)$ as testing $S$ by $\bar{x}$ to get a \textbf{t}hicket \textbf{est}imate of the measure of $S$.)

Write
\begin{equation} \label{treevariable}
	Y_{S, i}(\bar{x}) := \sum_{\tau \in 2^i} \left( 1_S( x_\tau ) \cdot \prod_{\sigma < \tau} \chi_S^\tau( x_\sigma) \right)
\end{equation}
so that $\test(\bar{x}, S) := \frac{1}{n} \sum_{i=0}^{n-1} Y_{S, i}(\bar{x})$. That is $Y_{S,i}(\bar{x}) = 1$ if the $i$th query is 1, following the above process. 

\begin{prop}[Weak law of large numbers]\label{weaklaw}
	Let $S \subseteq X$ be an event in a probability space $(X, \mu)$. Fix $\epsilon > 0$. Then for any integer $n$, we have
	\[
		\mu^n \left( \bar{x} \in X^n \, | \, \left| \frac{1}{n} \sum_{i=1}^n 1_S(x_i) - \mu(S) \right| \geq \epsilon \right) \leq \frac{1}{4n\epsilon^2}
	\]
\end{prop}
See \cite{simon2015guide} for a proof.

\begin{prop}[Weak law of large numbers for thickets]\label{weaklawthicket}
	Let $S \subseteq X$ be an event in a probability space $(X, \mu)$. Fix $\epsilon > 0$. Then for any integer $n$, we have
	\[
		\mu^{2^n-1} \left( (x_\sigma)_{\sigma \in 2^{<n}} \in X^{2^n-1} \, | \, \left| \frac{1}{n} \sum_{i=0}^{n-1} \sum_{\tau \in 2^i} \left( 1_S( x_\tau ) \cdot \prod_{\sigma < \tau} \chi_S^\tau( x_\sigma) \right) - \mu(S) \right| \geq \epsilon \right) \leq \frac{1}{4n\epsilon^2}.
	\]
\end{prop}

Intuitively, this follows from \ref{weaklaw} if we take
\[
	\frac{1}{n} \sum_{i=0}^{n-1} \sum_{\tau \in 2^i} \left( 1_S( x_\tau ) \cdot \prod_{\sigma < \tau} \chi_S^\tau( x_\sigma) \right) = \frac{1}{n} \sum_{i=0}^{n-1} 1_S(x_{\tau_{i,S}})
\] 
and view the random choices of the $x_\sigma$ as being made ``on the fly,'' and only making the choices which are relevant given our previous choices. We give a more precise proof.

\begin{proof}
	We follow the proof of \ref{weaklaw}. Observe that for each $i$, the random variable $Y_i: X^{2^n - 1} \rightarrow [0, 1]$,
	\[
		Y_i := \sum_{\tau \in 2^i} \left( 1_S( x_\tau ) \cdot \prod_{\sigma < \tau} \chi_S^\tau( x_\sigma) \right)
	\]
	has expectation $\mu(A)$ and variance $\mu(A)(1-\mu(A)) \leq \frac{1}{4}$, noting that $1_S( x_\tau ) \cdot \prod_{\sigma < \tau} \chi_S^\tau( x_\sigma)$ can be nonzero for exactly one $\tau \in 2^i$. Additionally, the variables $Y_0, \ldots Y_{n-1}$ are mutually independent. So $\frac{1}{n} \sum_{i=0}^{n-1} Y_i$ has expectation $\mu(A)$ and variance $\leq \frac{1}{4n}$. The conclusion then follows from Chebyshev's inequality.
\end{proof}

In other words, as the height of the test tree increases, the probability that a random test tree gives a good estimate of the measure of $S$ approaches 1. The adaptation of the VC-theorem states that, with probability approaching 1, a random test tree gives a uniformly good estimate of the measures of all $S \in \mathcal{F}$, provided that $\mathcal{F}$ has finite thicket dimension.

\begin{thm}[VC-theorem for thickets]
	Let $(X, \mu)$ be a finite probability space, $\mathcal{F} \subseteq \mathcal{P}(X)$ be a family of subsets, and $\rho_\mathcal{F}(n)$ be the thicket shatter function of $\mathcal{F}$. Then for any $n$ and $\epsilon > 0$, we have
	\begin{equation} \label{vcthmeq}
		\mu^{2^n-1} \left( \sup_{S \in \mathcal{F}} |\test(\bar{x}, S) - \mu(S) | > \epsilon  \right) \leq 8 \rho_\mathcal{F}(n) \exp\left(- \frac{n \epsilon^2}{32} \right).
	\end{equation}
\end{thm}

In particular, when $\mathcal{F}$ has finite thicket dimension, $\rho_\mathcal{F}(n)$ is bounded by a polynomial in $n$, and so the right-hand side of (\ref{vcthmeq}) approaches 0 exponentially quickly.

Our proof follows the proof of the VC-theorem closely---see \cite{simon2015guide}.

\begin{proof}
	Fix an integer $n$. For $\bar{x} = (x_\sigma)_{\sigma \in 2^{<n}}$, $\bar{x'} = (x'_\sigma)_{\sigma \in 2^{<n}}$, and $S \in \mathcal{F}$, let $f(\bar{x}, \bar{x'}, S) := |\test(\bar{x}, S) - \test(\bar{x'}, S)|$.
	
	Let $(x_\sigma)_{\sigma \in 2^{<n}}$, $(x'_\sigma)_{\sigma \in 2^{<n}}$ be mutually independent random elements from $X$, each with distribution $\mu$. Let $\xi_0, \ldots, \xi_{n-1}$ be random variables, independent from each other and from the previous ones, such that $\prob(\xi_i = 1) = \prob(\xi_i = -1) = \frac{1}{2}$, for all $i$.
	
	\begin{claimstar}
		We have
		\[
			\prob \left( \sup_{S \in \mathcal{F}} f(\bar{x}, \bar{x'}, S) > \epsilon/2 \right) \leq 2 \prob \left( \sup_{S \in \mathcal{F}} \frac{1}{n} \left| \sum_{i=0}^{n-1} \xi_i Y_{S,i}(\bar{x}) \right| \geq \epsilon/4  \right)
		\]
		with $Y_{S, i}$ as in (\ref{treevariable}).
	\end{claimstar}
	
	Observe that for fixed $i$ and $S$, the random variable $Y_{S, i}(\bar{x}) - Y_{S,i}(\bar{x}')$ has mean 0 and a symmetric distribution (taking values 1 and -1 with equal probability). Therefore its distribution does not change if we multiply it by $\xi_i$. We then compute:
	\begin{align*}
		& \prob \left( \sup_{S \in \mathcal{F}} f(\bar{x}, \bar{x'}, S) > \epsilon/2 \right) \\
		&= \prob \left( \sup_{S \in \mathcal{F}} \frac{1}{n} \left| \sum_{i=0}^{n-1} (Y_{S,i}(\bar{x}) - Y_{S,i}(\bar{x'})) \right| > \epsilon/2 \right) \\
		&= \prob \left( \sup_{S \in \mathcal{F}} \frac{1}{n} \left| \sum_{i=0}^{n-1} \xi_i (Y_{S,i}(\bar{x}) - Y_{S,i}(\bar{x'})) \right| > \epsilon/2 \right) \\
		&\leq \prob \left( \sup_{S \in \mathcal{F}} \frac{1}{n} \left| \sum_{i=0}^{n-1} \xi_i Y_{S, i}(\bar{x}) \right| > \epsilon/4 \text{ or } \sup_{S \in \mathcal{F}} \frac{1}{n} \left| \sum_{i=0}^{n-1} \xi_i Y_{S, i}(\bar{x'}) \right| > \epsilon/4 \right) \\
		&\leq 2 \prob \left( \sup_{S \in \mathcal{F}} \frac{1}{n} \left| \sum_{i=0}^{n-1} \xi_i Y_{S, i}(\bar{x}) \right| > \epsilon/4 \right).
	\end{align*}
	This proves the claim.
	
	\begin{claimstar}
		We have:
		\[
		\prob \left( \sup_{S \in \mathcal{F}} f(\bar{x}, \bar{x'}, S) > \epsilon/2 \right) \leq 4 \rho_\mathcal{F}(n) \exp \left(- \frac{n\epsilon^2}{32}\right).
		\]
	\end{claimstar}
	
	Fix a tuple $\bar{a} = (a_\sigma)_{\sigma \in 2^{<n}} \in X^{2^n - 1}$. For $S \in \mathcal{F}$, let $A_S(\bar{a})$ be the event $\frac{1}{2} |\sum_{i=0}^{n-1} \xi_i Y_{S,i}(\bar{a})| > \epsilon/4$, so that the only randomness is in the $\xi_i$'s. By Chernoff's bound (see Appendix B of \cite{simon2015guide}), we have
	\[
		\prob(A_S(\bar{a})) \leq 2 \exp \left( - \frac{n \epsilon^2}{32}\right).
	\]
	Note that $A_S(\bar{a})$ depends only on the characteristic path of $S$ through $\bar{a}$, of which there are at most $\rho_\mathcal{F}(n)$ values as $S$ ranges in $\mathcal{F}$. So there are at most $\rho_\mathcal{F}(n)$ events $A_S(\bar{a})$ to consider. Then the union bound shows that the disjunction $\cup_{S \in \mathcal{F}} A_S(\bar{a})$ has probability at most $2 \rho_\mathcal{F}(n) \exp(-n\epsilon^2/32)$. Then, by the first claim, we have
	\begin{align*}
		\prob \left( \sup_{S \in \mathcal{F}} f(\bar{x}, \bar{x'}, S) > \epsilon/2 \right) &\leq 2 \prob \left( \cup_{S \in \mathcal{F}} A_S(\bar{x}) \right) \\
		& \leq 4 \rho_\mathcal{F}(n) \exp \left( - \frac{n \epsilon^2}{32}\right).
	\end{align*}
	This proves the claim.
	
	We may assume that $n > 2/\epsilon^2$, as otherwise the right-hand side of (\ref{vcthmeq}) is larger than 1. Let
	\[
		X_0 := \{\bar{b} \in X^{2^n - 1} \, | \, \prob \left( \sup_{S \in \mathcal{F}} f(\bar{x}, \bar{b}, S) > \frac{\epsilon}{2} \right) \geq \frac{1}{2} \}.
	\]
	Then the second claim implies that $\mu^{2^n - 1}(X_0) \leq 8 \rho_\mathcal{F}(n) \exp(-n\epsilon^2/32)$. Fix $\bar{a} \in X^{2^n - 1} \setminus X_0$ and $S \in \mathcal{F}$. By Proposition \ref{weaklawthicket}, we have
	\[
		\prob \left(|\test(\bar{x}, S) - \mu(S)| > \frac{\epsilon}{2} \right) \leq \frac{1}{n\epsilon^2} < \frac{1}{2}.
	\]
	It follows that there is $\bar{x} \in X^{2^n - 1}$ such that $f(\bar{x}, \bar{a}, S) \leq \epsilon/2$ and $|\test(\bar{x}, S) - \mu(S)| \leq \epsilon/2$. This implies that $|\test(\bar{a}, S) - \mu(S)| \leq \epsilon$. As $S$ was arbitrary, we conclude that for any $\bar{a} \in X^{2^n - 1} \setminus X_0$, we have that $\sup_{S \in \mathcal{F}} |\test(\bar{a}, S) - \mu(S)| \leq \epsilon$. The result follows.
\end{proof}

The proof goes through verbatim if $X$ is infinite, provided we have some measurability conditions, namely:
\begin{itemize}
	\item each $S \in \mathcal{F}$ is measurable;
	\item for each $n$, the function 
	\[
		\bar{x} = (x_\sigma)_{\sigma \in 2^{<n}} \mapsto \sup_{S \in \mathcal{F}} |\test(\bar{x}, S) - \mu(S)|
	\]
	from $X^{2^n-1}$ to $\mathbb{R}$ is measurable; and
	\item for each $n$, the function
	\[
		(\bar{x}, \bar{x'}) = ((x_\sigma)_{\sigma \in 2^{<n}}, (x'_\sigma)_{\sigma \in 2^{<n}}) \mapsto \sup_{S \in \mathcal{F}} |\test(\bar{x}, S) - \test(\bar{x'}, S)|
	\]
	from $X^{2(2^n - 1)}$ to $\mathbb{R}$ is measurable.
\end{itemize}

To demonstrate the power of the result, we would like to find an example of a set system $\mathcal{F}$ over a probability space $(X, \mu)$ with finite VC-dimension but infinite thicket dimension such that
\[
	\mu^{2^n-1} \left( \sup_{S \in \mathcal{F}} |\test(\bar{x}, S) - \mu(S) | > \epsilon  \right)
\]
does not approach 0 exponentially quickly in $n$ (as it would if $\mathcal{F}$ had finite thicket dimension). This would show that sampling by test trees is meaningfully more complex than sampling by tuples, and that we can necessarily obtain uniformly good estimates of measures with probability exponentially approaching 1 only for set systems that have finite thicket dimension.

\section{Generalized banned sequence problems and applications}  \label{higherdim} 
In this section we generalize Theorem \ref{maincombin} to the setting of $j$-ary sequences, and apply the resulting combinatorics to prove Sauer-Shelah type lemmas in the $\op$-rank context \cite{guingonahill2015oprank}. 

\subsection{Banned $j$-ary sequence problems}

\begin{defn} 
	A \emph{$k$-fold banned $j$-ary sequence problem of length $n$} is a function 
	\[
		f: \binom{[n]}{k} \times j^{n-k} \rightarrow \mathcal{P} ( j ^k) \setminus \{\emptyset \}.
	\] 
	A \emph{solution} to $g$ is a $j$-ary sequence $X \in j^n$ such that for any $S \in \binom{[n]}{k}$, $$X|_S \notin f(S, X|_{[n] \setminus S} ).$$ 
\end{defn} 

As before, for a fixed $S \in \binom{[n]}{k}, $, we denote the elements of $S$ by $\{s_0, \ldots , s_{k-1} \}$, where $s_0 < s_1 < \ldots < s_{k-1}$. When we consider $f(S,X)$, we view $X \in j^{n-k}$ as a function $X:[n] \setminus S \rightarrow [j] = \{0, 1, \ldots, j-1\}$, and elements of $f(S,X)$ as functions $Z: S \rightarrow [j]$.

\begin{defn} 
	A $k$-fold banned $j$-ary sequence problem ($j$-ary BSP) of length $n$, $f$, is \emph{not hereditary} if there is $S \in \binom{[n]}{k}$ such that for all $Z_\alpha: S \rightarrow [j]$, there is $X_\alpha : [n] \setminus S \rightarrow [j]$ such that $Z_\alpha \notin f(S, X_\alpha)$, and additionally, for $Z_\alpha \neq Z_\beta$, the first index at which $X_\alpha \wedge Z_\alpha$ and $X_\beta \wedge Z_\beta$ differ is in $S$.
	
	Otherwise, say $f$ is \emph{hereditary}.
\end{defn}

\begin{thm} \label{gencombin} 
	Any hereditary $k$-fold banned $j$-ary sequence problem of length $n$ has at most $\sum_{i=0}^{k-1} (j-1)^{n-i}  \binom{n}{i}$ solutions.
\end{thm} 

\begin{proof} 
	The proof is by induction on $n$ and $k$, and is similar to the proof of Theorem \ref{maincombin}. Throughout, we fix a hereditary $k$-fold banned $j$-ary sequence problem of length $n$, $f$.
	
	Let $B(f)$ denote the number of sequences banned by $f$. It suffices to prove that
	\[
		B(f) \geq j^n - \sum_{i=0}^{k-1} (j-1)^{n-i}  \binom{n}{i}.
	\]
	
	The base cases are $k = n$ and $k = 1$. When $k = n$, the result is immediate, since then $j^n - \sum_{i=0}^{k-1} (j-1)^{n-i}  \binom{n}{i}$ = 1, and any $j$-ary BSP has at least one banned sequence.
	
	When $k = 1$, we must show that $B(f) \geq j^n - (j-1)^n$, i.e. there are at most $(j-1)^n$ solutions. Assume for contradiction that there are at least $(j-1)^n + 1$ solutions. We claim that there are $s \in [n]$ and solutions $X_0, \ldots, X_{j-1}$ such that $X_i(s) = i$ for all $i \in [j]$, and
	\[
		X_0|_{[s]} = X_1|_{[s]} = \cdots = X_{j-1}|_{[s]}.
	\]
	Such an $s$ and $X_i$ witness that $f$ is not hereditary, a contradiction. If we cannot find such $X_i$ for $s = 0$, i.e. there is some $m_0 \in [j]$ such that for no solution $X$ does $X(0) = m_0$, then by the generalized pigeonhole principle there is some $l_0 \in [j]$ such that for at least $(j-1)^{n-1} + 1$ solutions $X$, we have $X|_{[1]} = l_0$. From among these sequences, if we cannot find such $X_i$ for $s = 1$, i.e. there is some $m_1 \in [j]$ such that for no solution $X$ does $X(1) = m_1$, then by the generalized pigeonhole principle there is some $l_1 \in [j]$ such that for at least $(j-1)^{n-2} + 1$ solutions $X$, we have $X|_{[2]} = l_0 l_1$. Continue in this fashion. If at any index $s$ we find the desired solutions $X_i$, we are done. If this process does not terminate earlier, we eventually obtain at least $(j-1) + 1 = j$ solutions $X$ with $X|_{[n-1]} = l_0 l_1 \cdots l_{n-2}$, giving us the desired solutions $X_i = l_0 l_1 \cdots l_{n-2} i$ for each $i \in [j]$, with $s = n-1$. This proves the $k = 1$ case.
	
	Otherwise, we proceed by induction. Given a $j$-ary sequence $X$ of length $n$, we call an entry of $X$ \emph{bad} if it is part of a $k$-subset $S$ such that $X|_S \in f(S, X|_{[n] \setminus S})$. 
	
	
	We call a $j$-ary sequence of length $n-1$, $X'$ \emph{potentially bad for $T \in \binom{[n-1]}{k-1}$} if there is $l \in [j]$ such that the $(n-1)$th entry of $X \wedge l$ is bad for $S = T \cup \{n-1\}$. That is, setting $S = T \cup \{n-1\}$, we have $X'|_T \wedge l \in f(S, X'|_{[n] \setminus T})$. 
	
	Let $\hat{f}$ be the $(k-1)$-fold banned $j$-ary sequence problem of length $n-1$ given by those sequences in $j^{n-1}$ which are potentially bad for some $T \in \binom{[n-1]}{k-1}$. That is, let $Z \in \hat{f}(T, X)$ if $Z \wedge X$ is potentially bad for $T$. We claim that $\hat{f}$ is hereditary. If $\hat{f}$ were not hereditary, then there would exist $T \in \binom{[n-1]}{k-1}$ such that for each $Z_\alpha: T \rightarrow [j]$, there is $X_\alpha: [n-1] \setminus T \rightarrow [j]$ such that $Z_\alpha \notin \hat{f}(T, X_\alpha)$, with the first index at which any two $X_\alpha \wedge Z_\alpha$ and $X_\beta \wedge Z_\beta$ differ belonging to $T$. Note that for some $\alpha$ and some $l \in [j]$, we have that
	\[
		Z_\alpha \wedge l \in f(T \cup \{n-1\}, X_\alpha),
	\]
	or else associating $Z_\alpha \wedge l$ with $X_\alpha$ would witness that $f$ itself is not hereditary. In particular, $Z_\alpha \wedge X_\alpha$ is potentially bad for $T$. Then, by definition of $\hat{f}$, $Z_\alpha \in \hat{f}(T, X_\alpha)$, a contradiction. So $\hat{f}$ is also hereditary.
	
	Let $f'$ be the $k$-fold banned $j$-ary sequence problem of length $n-1$ given by those sequences which contain a banned subsequence for $f$ in the first $n-1$ entries, for any choice of the last entry. That is, given $S \in \binom{[n-1]}{k}$, let $Z \in f'(S, X)$ iff $Z \in f(S, X \wedge l)$ for all $l \in [j]$. We claim that $f'$ is hereditary. Suppose for contradiction that $f'$ is not hereditary. Then there is $S \in \binom{[n-1]}{k}$ such that for all $Z_\alpha: S \rightarrow [j]$, there is $X_\alpha: [n-1] \setminus S \rightarrow [j]$ such that $Z_\alpha \notin f'(S, X_\alpha)$, with the first index at which any two $X_\alpha \wedge Z_\alpha$ and $X_\beta \wedge Z_\beta$ differ belonging to $S$. That is, by definition of $f'$, for each $\alpha$, there is $l_\alpha \in [j]$ such that $Z_\alpha \notin f(S, X_\alpha \wedge l_\alpha)$. Let $X'_\alpha$ be $X_\alpha \wedge l_\alpha$. Then associating $Z_\alpha$ with $X'_\alpha$ witnesses that $f$ is not hereditary, a contradiction.
	
	Now, we aim to prove
	\[
		B(f) \geq B(\hat{f}) + B(f') \cdot (j-1).
	\]
	For a given sequence $X'$ which is banned by $\hat{f}$, there is at least one extension $X$ which is banned by $f$, and we pick one such extension. For a given sequence $Y'$ banned by $f'$, at most one extension $Y$ was already obtained by extending a banned sequence $X'$ for $\hat{f}$. So there are at least $j-1$ extensions $Y$ which are banned by $f$ (by definition of $f'$) but was not obtained by extending banned sequences for $\hat{f}$. Therefore these banned sequences constructed from $f'$ and $\hat{f}$ have no common members, and so we have 
	\[
		B(\hat{f}) + B(f') \cdot (j-1),
	\] 
	as desired. By induction, we have that
	\begin{align*}
		B(f)  & \geq j^{n-1} - \sum _{i=0}^{k-2} (j-1)^{n-1-i}  \binom{n-1}{i} \\
		& \quad + \left( j-1 \right) \left(j^{n-1} - \sum_{i=0}^{k-1} (j-1)^{n-1-i}  \binom{n-1}{i} \right) \\ 
		& = j^n - \sum_{i=0}^{k-1} (j-1)^{n-i}  \binom{n}{i} .
	\end{align*} 
	Thus, $f$ has at most $\sum_{i=0}^{k-1} (j-1)^{n-i}  \binom{n}{i}$ solutions.
	
\end{proof}

\subsection{On the op-rank shatter function} \label{opsection}

The context of banned $j$-ary sequences allows us to work in the $\op$-rank context of \cite{guingonahill2015oprank}, which we reframe in terms of set systems. Whereas VC dimension and thicket dimension make use of binary trees, $\op_s$-rank makes use of $2^s$-ary trees.

\begin{defn}
	A \emph{$2^s$-ary element tree} $T$ of height $n$ with labels from $X$ is a labeling of each node $\nu \in (2^s)^{<n}$ by $s$-tuples $x_\nu = (x_{\nu,0}, \ldots, x_{\nu, s-1})$ from $X$. A \emph{leaf} of $T$ is an element of $(2^s)^n$. A leaf $\xi$ is properly labeled by a set $A$ if, for all $j < n$ and for all $i < s$, $x_{\xi|_{[j]}, i} \in A $ iff $\xi(j)(i) = 1$.
\end{defn}

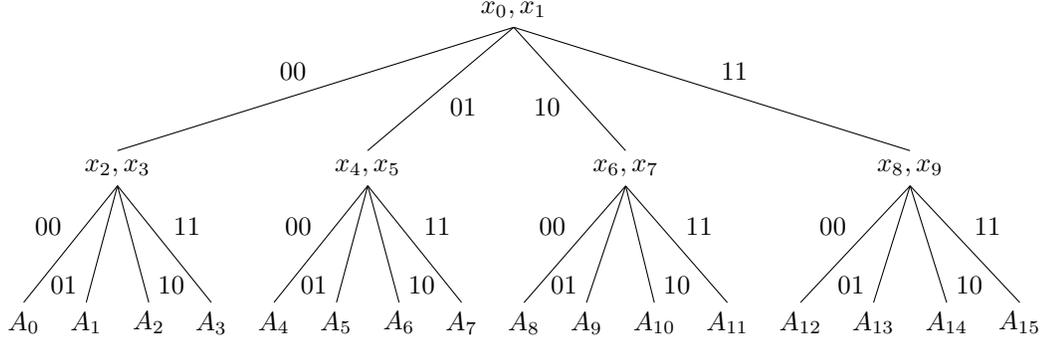
\begin{figure}[h] 
	\centering
	\begin{tikzpicture}
	\tikzset{level distance=60pt}
	\tikzset{sibling distance=5pt}
		\Tree [.{$x_0, x_1$} \edge node[auto=right]{00}; [.{$x_2, x_3$} \edge 	node[auto=right]{00}; $A_0$ \edge node[auto=right, pos=1]{01}; $A_1$ \edge node[auto=left, pos=1]{10}; $A_2$ \edge node[auto=left]{11}; $A_3$ ]
			\edge node[auto=left]{01}; [.{$x_4, x_5$}  \edge node[auto=right]{00}; $A_4$ \edge node[auto=right, pos=1]{01}; $A_5$ \edge node[auto=left, pos=1]{10}; $A_6$ \edge node[auto=left]{11}; $A_7$ ]
			\edge node[auto=right]{10}; [.{$x_6, x_7$} \edge node[auto=right]{00}; $A_8$ \edge node[auto=right, pos=1]{01}; $A_9$ \edge node[auto=left, pos=1]{10}; $A_{10}$ \edge node[auto=left]{11}; $A_{11}$ ]
			\edge node[auto=left]{11}; [.{$x_8, x_9$} \edge node[auto=right]{00}; $A_{12}$ \edge node[auto=right, pos=1]{01}; $A_{13}$ \edge node[auto=left, pos=1]{10}; $A_{14}$ \edge node[auto=left]{11}; $A_{15}$ ] ]
	\end{tikzpicture}
	\caption{A $2^2$-ary element tree of height 2. $A_9$ properly labels its leaf if it contains $x_0$ and $x_7$, but does not contain $x_1$ and $x_6$, with no requirements on membership of the other elements.}
	\label{opfig1}
\end{figure}

While this will be the definition that we use in practice, it is often useful think of such trees as binary trees with certain requirements on uniformity of labels within levels.

\begin{defn}
	An \emph{alternative $2^s$-ary element tree} $T$ of height $n$ with labels from $X$ is a labeling of $2^{<ns}$ by elements of $X$ such that given any two nodes $\sigma$ and $\sigma'$ with labels $x_{\sigma}$ and $x_{\sigma'}$,  if $|\sigma| = |\sigma'| = l$ and $\sigma|_{s[ \lfloor \frac{l}{s} \rfloor ]} = \sigma'|_{s[ \lfloor \frac{l}{s} \rfloor ]}$, then $x_\sigma = x_{\sigma'}$. A \emph{leaf} of $T$ is an element of $2^{ns}$, i.e. a binary sequence of length $ns$. A leaf $\tau$ is properly labeled by a set $A$ if, for all $j < ns$, $x_{\tau|_{[j]}} \in A$ iff $\tau(j) = 1$.
\end{defn}

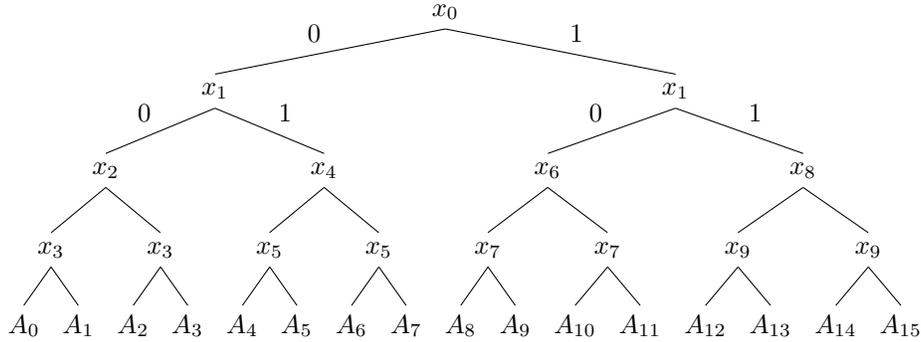
\begin{figure}[h] 
	\centering
	\begin{tikzpicture}
		\Tree [.$x_0$ \edge node[auto=right]{0}; [.$x_1$ \edge node[auto=right]{0}; [.$x_2$ [.$x_3$ $A_0$ $A_1$ ]
					[.$x_3$ $A_2$ $A_3$ ] ]
		  		\edge node[auto=left]{1}; [.$x_4$ [.$x_5$ $A_4$ $A_5$ ]
					[.$x_5$ $A_6$ $A_7$ ] ] ]
			\edge node[auto=left]{1}; [.$x_1$ \edge node[auto=right]{0}; [.$x_6$ [.$x_7$ $A_8$ $A_9$ ]
					[.$x_7$ $A_{10}$ $A_{11}$ ] ]
				\edge node[auto=left]{1}; [.$x_8$ [.$x_9$ $A_{12}$ $A_{13}$ ]
					[.$x_9$ $A_{14}$ $A_{15}$ ] ] ] ]
	\end{tikzpicture}
	\caption{An alternative $2^2$-ary element tree of height 2. Observe that the labels on the first two levels are uniform. Then, on the fourth level (and, trivially, the third level), labels are uniform across all nodes with the same initial segment of length 2. We identify 1 with the right branch. As before, $A_9$ properly labels its leaf if it contains $x_0$ and $x_7$, but does not contain $x_1$ and $x_6$, with no requirements on membership of the other elements.}
	\label{opfig2}
\end{figure}

\begin{defn}
	The \emph{$\op_s$-rank} of a set system $(X, \mathcal{F})$, written $\opR_s(X, \mathcal{F})$ or $\opR_s(\mathcal{F})$, is the largest $k < \omega$ such that there is a $2^s$-ary element tree of height $k$ with labels from $X$ such that every leaf can be properly labeled by elements of $\mathcal{F}$, or is infinite if there are such trees of arbitrary height. As a convention, we set $\opR_s(\mathcal{F}) = -\infty$ if $\mathcal{F} = \emptyset$. The \emph{$\op_s$ shatter function} $\psi_\mathcal{F}^s(n)$ is the maximum number of leaves properly labeled by elements of $\mathcal{F}$ in a $2^s$-ary element tree of height $n$.
\end{defn}

It is easy to verify that the $\op_s$-rank and $\op_s$ shatter function do not depend on which definition of $2^s$-ary element tree we take.

The $\op_s$ context is therefore intermediate between thicket context and VC context. Instead of picking labels node by node (as in the thicket context) or uniformly for a single level (as in the VC context), we pick labels $s$ at a time. As before, the dual $\op_s$-rank and dual $\op_s$ shatter function of a set system are the $\op_s$-rank and $\op_s$ shatter function of the dual set system.

\begin{cor}\label{basicopshatter}
	Let $\mathcal{F}$ be a set system with $\opR_s(\mathcal{F}) = k$. Then 
	\[
		\psi^s_\mathcal{F}(n) \leq \sum _{i=0}^k (2^s-1)^{n-i}  \binom{n}{i}.
	\] 
\end{cor}

The proof follows our proof of Theorem \ref{thicketSS}, using $j$-ary banned sequence problems.

\begin{proof}
	Let $T$ be an $2^s$-ary element tree of height $n$. Identifying the $2^s$ binary sequences of length $s$ with $[2^s]$, we obtain a hereditary $(k+1)$-fold banned $2^s$-ary sequence problem of length $n$, $f$, as follows. Given $S = \{s_0, \ldots, s_k\} \in \binom{[n]}{k+1}$, where $s_0 < s_1 < \cdots < s_k$ and $X \in (2^s)^{[n] \setminus S}$, we obtain a $2^s$-ary element tree of height $k+1$ by taking all paths $\xi \in (2^s)^n$ through $T$ such that $X \subset \tau$. Decisions will only be made at nodes $\nu$, where $|\nu| \in S$, so removing all other nodes gives us a $2^s$-ary element tree $T_{S, X}$ of height $k+1$. Since $\opR_s(\mathcal{F}) = k$, not all leaves of $T_{S,X}$ can be properly labeled, so let $f(S, X)$ be the set of all sequences whose corresponding leaves in $T_{S, X}$ cannot be properly labeled. Then a leaf in $T$ can only be properly labeled if the corresponding sequence is a solution to $f$.
	
	It remains to show that $f$ is hereditary. Fix $S = \{s_0, \ldots, s_k\}$, and suppose for contradiction that this choice of $S$ witnesses that $f$ is \emph{not} hereditary. Then, for any $Z_\alpha: S \rightarrow \{0, \ldots, 2^s - 1\}$, there is $X_\alpha: [n] \setminus S \rightarrow \{0, \ldots, 2^s - 1\}$ such that $Z_\alpha \notin f(S, X_\alpha)$. We obtain a complete $2^s$-ary tree of height $k+1$ specified by each path $X_\alpha \wedge Z_\alpha$ constructed in this manner, restricted to $S$. Since each $Z_\alpha$ is not banned, we have a $2^s$-ary tree of height $k+1$ in which every leaf can be properly labeled, a contradiction.
	
	The result then follows from Theorem \ref{gencombin}.
\end{proof}

The bound of Corollary \ref{basicopshatter} can be improved by using more information---in particular, when bounding the $\op_s$ shatter function, we can consider $\op_r$-ranks for $r \leq s$. We can already give a better bound for the case where a set system has $\op_r$-rank 0 for some $r$.

\begin{prop} \label{oprank0}
	Let $\mathcal{F}$ be a set system with $\opR_r(\mathcal{F}) = 0$. Then 
	\[
		\psi^s_\mathcal{F}(n) \leq \left( \sum_{i=0}^{r-1} \binom{s}{i} \right)^n
	\]
\end{prop}

\begin{proof}
	Call a node \emph{live} if it the initial segment of a leaf that can be properly labeled. At each node of the tree, we consider $s$ elements. Observing that $\opR_r(\mathcal{F}) = 0$ says precisely that the VC dimension of $\mathcal{F}$ is strictly less than $r$, Theorem \ref{VC-SS} tells us that we can find sets which properly label at most $\sum_{i=0}^{r-1} \binom{s}{i}$ of the possible boolean combinations of the $s$ elements. That is, each live node has at most $\sum_{i=0}^{r-1} \binom{s}{i}$ live successors in the next level. Therefore, there are at most $\left( \sum_{i=0}^{r-1} \binom{s}{i} \right)^m$ live nodes at the level of height $m$ (counting from 0). Since leaves in a tree of height $n$ appear at the $n$th level, the result follows.
\end{proof}

The set system of half-spaces in $\mathbb{R}^r$ achieves the bound of Proposition \ref{oprank0} for the \emph{dual} $\op_s$ shatter function. (This is the famous cake-cutting problem.) 

\begin{prop}
	Let $\mathcal{F}$ be the dual set system to the set system of $\mathbb{R}^r$ consisting of half-spaces. Then
	\[
		\psi^s_\mathcal{F}(n) = \left( \sum_{i=0}^r \binom{s}{i} \right)^n.
	\]
	In particular, $\opR_{r+1}(\mathcal{F}) = 0$.
\end{prop}

\begin{proof}
	It suffices to verify that taking $s$ hyperplanes in general position (i.e. so that any $m$ hyperplanes intersect in a $(r-m)$-dimensional subspace) partitions $\mathbb{R}^r$ into $\sum_{i=0}^r \binom{s}{i}$ pieces, each of which contains an open set (in the Euclidean topology). Such a partition corresponds to one level in the $2^s$-ary tree. Each piece may then be partitioned further in the same manner for each successive level of the tree.
	
	We proceed by induction. The $s = 1$ case is obvious, for all $r$. The $r = 1$ case is obvious, for all $s$.
	
	Consider the $s+1$ and $r+1$ case. Removing one of the $s+1$ hyperplanes, we have $\sum_{i=0}^{r+1} \binom{s}{i}$ pieces by induction. Restore the hyperplane that we removed. Viewing that hyperplane as a copy of $\mathbb{R}^r$, it is partitioned into $\sum_{i=0}^r \binom{s}{i}$ pieces by the other hyperplanes, by induction. Each such piece corresponds to a piece in $\mathbb{R}^{r+1}$ which is cut into two pieces by the restored hyperplane. We therefore find that the total number of pieces is
	\[
		\sum_{i=0}^{r+1} \binom{s}{i} + \sum_{i=0}^r \binom{s}{i} = \sum_{i=0}^{r+1} \binom{s+1}{i}.
	\]
	as desired.
\end{proof}

We can further refine our methods further. Fix a base set $X$. We identify any set system $(X, \mathcal{F})$ on $X$ with $\mathcal{F}$.

\begin{prop} \label{opproperties}
	\begin{enumerate}
		\item Let $\mathcal{F}_1 \subseteq \mathcal{F}_2$. Then, for any $s$, $\opR_s(\mathcal{F}_1) \leq \opR_s(\mathcal{F}_2)$.
		\item Let $s_1 < s_2$. Then $\opR_{s_1}(\mathcal{F}) \geq \lfloor \frac{s_2}{s_1} \rfloor \opR_{s_2}(\mathcal{F})$.
	\end{enumerate}
\end{prop}

\begin{proof}
	(1) is trivial. For (2), suppose that we have a $2^{s_2}$-ary element tree $T$ of height $n_2 := \opR_{s_2}(\mathcal{F})$, with labels $x_\nu = (x_{\nu, 0}, \ldots, x_{\nu, s_2 - 1})$ for each $\nu \in (2^{s_2})^{<n_2}$, in which every leaf can be properly labeled. Then we can obtain a $2^{s_1}$-ary element tree $T'$ of height $n_1 := \lfloor \frac{s_2}{s_1} \rfloor n_2$ in which every leaf can be properly labeled. Let $t = \lfloor \frac{s_2}{s_1} \rfloor$. Intuitively, we split each level of the $2^{s_2}$-ary tree into $t$ levels of the $2^{s_1}$-ary tree, with any label $x_\nu = (x_{\nu, 0}, \ldots, x_{\nu, s_2 - 1})$ splitting into $t$ labels 
	\[
		(x_{\nu, 0}, \ldots, x_{\nu, s_1 - 1}), (x_{\nu, s_1}, x_{\nu, 2s_1 - 1}), \ldots, (x_{\nu, (t-1)s_1}, \ldots, x_{\nu, ts_1 - 1}).
	\] 
	More formally, suppose $\xi \in (2^{s_1})^i$, for $i < n_1$. Suppose $i = jt + k$, for $0 \leq k < t$. Then label $\xi$ with
	\[
		x_\xi = (x_{\nu_\xi, ks_1}, \ldots x_{\nu_\xi, (k+1)s_1 - 1}),
	\]
	where $\nu_\xi \in (2^{s_2})^j$ is as follows. Let $\sigma_l = \xi(l) \in 2^{s_1}$. Then let $\tau_m \in 2^{s_2}$ be the concatenation of $\sigma_{mt}, \ldots, \sigma_{(m+1)t - 1}$, appending as many 0s as needed to obtain a sequence of length $s_2$. Then let 
	\[
		\nu_\xi := (\tau_0, \ldots, \tau_{j-1}).
	\]
	Then the labeling of $T'$ by the $x_\xi$ gives a $2^{s_1}$-ary tree of height $n_1$ in which every leaf can be properly labeled (in particular, by one of the labels of the leaves of the $2^{s_2}$-ary tree).
\end{proof}

(2) above is somewhat easier to see using the alternative definition---we simply view the tree as a $2^{s_1}$-ary tree instead of a $2^{s_2}$-ary tree, possibly after removing some levels. Figure \ref{opfig2} is the $2^1$-ary tree obtained from Figure \ref{opfig1} by this process.

Given $\mathcal{F}$, $x_0, \ldots, x_{s-1} \in X$, and $\sigma: [s] \rightarrow 2$, let 
\[
	\mathcal{F}_\sigma := \{Y \in \mathcal{F} \, | \, \text{for all } i < n,\, x_i \in Y \text{ iff } \sigma(i) = 1 \}.
\]
Call each $\mathcal{F}_\sigma$ a child of $\mathcal{F}$. Then, in an $\op_s$-tree with root $(x_0, \ldots, x_{s-1})$, $\mathcal{F}_\sigma$ consists of all sets in $\mathcal{F}$ which properly label a leaf whose path begins with $\sigma$. Observe that if for all $\sigma: [s] \rightarrow 2$, $\opR_s(\mathcal{F}_\sigma) \geq a$, then $\opR_s(\mathcal{F}) \geq a + 1$; we can obtain a $2^s$-ary tree of height $a+1$ by labeling the root with $(x_0, \ldots, x_{s-1})$, and appending $2^s$-ary trees of height $a$ witnessing $\opR_s(\mathcal{F}_\sigma) \geq a$ at the appropriate successor nodes.

\begin{lem} \label{opchild}
	Suppose $\opR_r(\mathcal{F}) = a < \infty$. Then, given any $x_0, \ldots, x_{s-1} \in X$, we have $\opR_r(X_\sigma) \leq a - 1$ for at least $2^s - \sum_{i=0}^{r-1} \binom{s}{i}$ children $\mathcal{F}_\sigma$. More generally, we have $\opR_r(X_\sigma) \leq a - l$ for at least $2^s - \sum_{i=0}^{lr - 1} \binom{s}{i}$ children $\mathcal{F}_\sigma$.
\end{lem}

\begin{proof}
	We obtain an independent $r$-fold banned binary sequence problem $f$ of length $s$ as follows. For each $S \in \binom{[s]}{r}$, let $f(S)$ be those functions $\eta: S \rightarrow 2$ such that $\opR_r(\mathcal{F}_\eta) \leq a - 1$, where 
	\[
		\mathcal{F}_\eta := \{Y \in \mathcal{F} \, | \, \text{for all } i \in S ,\, x_i \in Y \text{ iff } \eta(i) = 1 \}.
	\]
	Each $f(S)$ is nonempty, or else those $X_\eta$ witness that that $\opR_r(\mathcal{F}) \geq a + 1$, a contradiction. Then $\sigma: [s] \rightarrow 2$ is banned by $f$ if there is some $S \in \binom{[s]}{r}$ such that $\opR_r(\mathcal{F}_{\sigma|_S}) \leq a - 1$, whence $\opR_r(\mathcal{F}_\sigma) \leq a - 1$. So sequences banned by $f$ have the corresponding child drop in $\op_r$-rank, of which there are at least $2^s -\sum_{i=0}^{r-1} \binom{s}{i}$ many.
	
	For the more general case, we instead obtain an independent $lr$-fold banned binary sequence problem. For each $S \in \binom{[s]}{lr}$, let $f(S)$ be those $\eta: S \rightarrow 2$ such that $\opR_r(\mathcal{F}_\eta) \leq a - l$. Each $f(S)$ is nonempty, or else those $\mathcal{F}_\eta$ witness that $\opR_r(\mathcal{F}) \geq a + 1$. Then sequences banned by $f$ have the corresponding child drop in $\op_r$-rank by at least $l$, of which there are at least $2^s - \sum_{i=0}^{lr - 1} \binom{s}{i}$ many.
\end{proof}

The boundary between finite and infinite $\op$-ranks serves as an important parameter in obtaining better bounds. It is also of model-theoretic interest, coinciding with other known properties.

\begin{defn}
	The \emph{$\op$-dimension} of a set system $\mathcal{F}$ is 
	\[
		\sup \{r \, | \, \opR_r(\mathcal{F}) = \infty \}.
	\]
\end{defn}

Expressed in model-theoretic terms, the $\op$-dimension of a (type-)definable set $X$ in some model is the supremum of the $\op$-dimension of set systems on $X$ generated finite sets of formulas. In this context, $\op$-dimension coincides with o-minimal dimension in o-minimal theories and dp-rank in distal theories \cite{guingonahill2015oprank}.

We use Lemma \ref{opchild} to obtain better bounds on the $\op_s$ shatter function by using $\op$-dimension.

\begin{defn}
	Let $\psi_{r,b}^s(n)$ be the greatest possible number of properly labeled leaves in a $2^s$-ary tree of height $n$ by any set system $\mathcal{F}$ with $\opR_r(\mathcal{F}) \leq b < \omega$.
\end{defn}

\begin{thm}
	Let $a_0 := \sum_{i=0}^{r-1} \binom{s}{i}$ and $a_1 = 2^s - a_0$. Then
	\[
		\psi_{r,b}^s(n) \leq \sum_{i=0}^b \binom{n}{i} a_0^{n - i} a_1^i.
	\]
\end{thm}

\begin{proof}
	The case $n = 0$ is trivial for all $b$. We proceed by induction on $b$. The case $b = 0$ is Proposition \ref{oprank0}.
	
	For the case $b + 1$, we observe that, by monotonicity of $\psi_{r,b}^s(n)$ in $b$, we maximize the possible number of properly labeled leaves by having as many children as possible not decrease in $\op_r$-rank. We now proceed by induction on $n$. By Lemma \ref{opchild}, we can have at most $a_0$ such children, and the remaining $a_1$ children must drop in $\op_r$-rank by at least 1. We therefore obtain the recurrence
	\begin{align*}
		\psi_{r,b+1}^s(n) & \leq a_0 \psi_{r,b+1}^s(n-1) + a_1 \psi_{r,b}^s(n-1) \\
		& \leq a_0 \sum_{i=0}^{b+1} \binom{n-1}{i} a_0^{n - i - 1} a_1^i + a_1 \sum_{i=0}^b \binom{n-1}{i} a_0^{n - i - 1} a_1^i \text{\quad by induction}\\
		& \leq \sum_{i=0}^{b+1} \binom{n-1}{i} a_0^{n - i} a_1^i + \sum_{i=0}^b \binom{n-1}{i} a_0^{n - i - 1} a_1^{i+1} \\
		& \leq \binom{n-1}{0} a_0^n + \sum_{i=1}^{b+1} \binom{n-1}{i} a_0^{n - i} a_1^i + sum_{i=1}^{b+1} \binom{n-1}{i-1} a_0^{n-i} a_1^i \\
		& \leq \binom{n}{0} a_0^n + \sum_{i=1}^{b+1} \binom{n}{i} a_0^{n - i} a_1^i \\
		& \leq \sum_{i=0}^{b+1} \binom{n}{i} a_0^{n - i} a_1^i
	\end{align*}
	as desired.
	
\end{proof}

In particular, for a set system with $\op$-dimension $d$, we take $r = d+1$. Then the $\op$ shatter function is bounded by an exponential function with the base $a_0$ determined by $d$. Furthermore, coefficients for lower order terms can be improved when $r \leq \frac{s}{2}$, as then the more general case of Lemma \ref{opchild} dictates that some children must drop in $\op_r$-rank by more than 1. This creates a more complicated recurrence, but the result remains exponential in $a_0$.

\bibliography{Research}{}

\begin{thebibliography}{1}

\bibitem{bhaskar2017thicket}
Siddharth Bhaskar.
\newblock Thicket density.
\newblock {\em arXiv preprint arXiv:1702.03956}, 2017.

\bibitem{MLMT}
Hunter Chase and James Freitag.
\newblock Model theory and machine learning.
\newblock {\em In preparation}, 2018.

\bibitem{chernikov2014n}
Artem Chernikov, Daniel Palacin, and Kota Takeuchi.
\newblock On n-dependence.
\newblock {\em arXiv preprint arXiv:1411.0120}, 2014.

\bibitem{guingonahill2015oprank}
Vincent Guingona and Cameron~Donnay Hill.
\newblock On a common generalization of {S}helah's 2-rank, dp-rank, and
  o-minimal dimension.
\newblock {\em Annals of Pure and Applied Logic}, 166(4):502--525, 2015.

\bibitem{laskowski1992vapnik}
Michael~C Laskowski.
\newblock Vapnik-{C}hervonenkis classes of definable sets.
\newblock {\em Journal of the London Mathematical Society}, 2(2):377--384,
  1992.

\bibitem{MS2014stableregularity}
Maryanthe Malliaris and Saharon Shelah.
\newblock Regularity lemmas for stable graphs.
\newblock {\em Transactions of the American Mathematical Society},
  366:1551--1585, 2014.

\bibitem{MT2017}
Maryanthe Malliaris and Caroline Terry.
\newblock On unavoidable induced subgraphs in large prime graphs.
\newblock {\em Journal of Graph Theory, accepted}, 2017.

\bibitem{ngo2010SSproofs}
Hung~Q. Ngo.
\newblock Three proofs of {S}auer-{S}helah lemma.
\newblock {\em Course notes,
  \url{https://www.cse.buffalo.edu/~hungngo/classes/2010/711/lectures/sauer.pdf}},
  2010.

\bibitem{simon2015guide}
Pierre Simon.
\newblock {\em A guide to NIP theories}.
\newblock Cambridge University Press, 2015.

\end{thebibliography}
\bibliographystyle{plain}

\end{document}